\documentclass[10pt]{article}

\usepackage[utf8]{inputenc}
\usepackage[T1]{fontenc}
\usepackage{amsmath,amsfonts,amssymb,amsthm,latexsym,graphicx}
\usepackage{a4wide}
\usepackage{hyperref}
\usepackage{here,multicol,eso-pic,rotating}

\newcommand{\R}{\mathbb{R}}
\newcommand{\N}{\mathbb{N}}

\newtheorem{theorem}{Theorem}
\newtheorem{lemma}[theorem]{Lemma}

\newtheorem{proposition}[theorem]{Proposition}

\title{Geodesic Normal distribution on the circle}

\author{Jean-François Coeurjolly$^{1,2}$ and Nicolas Le Bihan$^{2,3}$ \\ 
$^1$ Laboratoire Jean Kuntzmann, Grenoble University, France. \\
$^2$ GIPSA-lab, Grenoble University, France, $^3$ CNRS, France.}
\date{\today}
\begin{document}

\maketitle

\begin{center}
\begin{minipage}{12cm}
\centerline{\bf Abstract}
This paper is concerned with the study of a circular random distribution called geodesic Normal distribution recently proposed for general manifolds. This distribution, parameterized by two real numbers associated to some specific location and dispersion concepts, looks like a standard Gaussian on the real line except that the support of this variable is $[0,2\pi)$ and that the Euclidean distance is replaced by the geodesic distance on the circle. Some properties are studied and comparisons with the von~Mises distribution in terms of intrinsic and extrinsic means and variances are provided. Finally, the problem of estimating the parameters through the maximum likelihood method is investigated and illustrated with some simulations. 
\end{minipage}
\end{center}

\bigskip
Circular statistics deal with random variables taking values on hyperspheres and can be included in the broader field of directional statistics. Applications of circular statistics are numerous and can be found, for example, in fields such as climatology (wind direction data \cite{B-MarJup00}), biology (pigeons homing performances \cite{Wat83}) or earth science (earthquake locations occurence and other data types, see~\cite{B-MarJup00} for examples) among others.

A circular distribution is a probability distribution function (pdf) which mass is concentrated on the circumference of a unit circle. The support of a random variable 
$\theta$ representing an angle measured in radians may be taken to $[0,2\pi)$ or $[-\pi,\pi)$. We will focus here on continuous circular distributions,
that is on absolutely continuous (w.r.t. the Lebesgue measure on the circumference) distributions. A pdf of a circular random variable has to fulfill
the following axioms
\begin{itemize}
\item[(i)] $f(\theta)\geq 0$.
\item[(ii)] $\int_0^{2\pi} f(\theta)d\theta =1$
\item[(iii)] $f(\theta)=f(\theta +2 k \pi)$ for any integer $k$ (i.e. $f$ is periodic).
\end{itemize}

Among many models of circular data, the von~Mises distribution plays a central role (essentially due to the similarities shared with the Normal 
distribution on the real line). A circular random variable (for short r.v.) $\theta$ is said to have a von~Mises distribution, denoted by $vM(\mu,\kappa)$, if it has the density function
$$
f(\theta; \mu,\kappa) = \frac1{2 \pi I_0(\kappa)} e^{\kappa \cos(\theta-\mu)},
$$
where $\mu \in [0,2\pi)$ and $\kappa\geq 0$ are parameters and where $I_0(\kappa)$ is the modified Bessel function of order~0. The aim of this paper is to review some properties of another circular 
distribution introduced by \cite{A-Pen06} which also shares similarities with the Normal distribution on the real line. For some $\mu\in[0,2\pi)$ and for some parameter $\gamma\geq 0$, a r.v. $\theta$ is said to have a geodesic Normal distribution denoted in the following $gN(\mu,\gamma)$, if it has the density function
$$
f(\theta;\mu,\gamma) = k^{-1}(\gamma) e^{-\frac{\gamma}2 d_G(\mu,\theta)^2},
$$
where $d_G(\cdot,\cdot)$ is the geodesic distance on the circle and where $k(\gamma)$ is the normalizing constant defined by
$$k(\gamma):=\sqrt{\frac{2\pi}\gamma} \left( \Phi(\pi\sqrt{\gamma})-\Phi(-\pi\sqrt{\gamma}\right) = \sqrt{\frac{2\pi}\gamma} erf\left(\pi \sqrt{\frac\gamma 2} \right),$$ 
where $\Phi$ is the cumulative distribution function of a standard Gaussian random variable and where $erf(\cdot)$ is the error function. \\
Let us underline that Pennec \cite{A-Pen06} introduced the geodesic Normal distribution for general Riemannian manifolds. We focus here on a special (and simple) manifold, the circle, in order to highlight its basic properties and compare them with the most classical circular distribution, namely the $vM(\mu,\kappa)$ distribution. That is, we present here a new study of the $gN$ distribution in the framework of circular statistics and provide results in terms of estimation and asymptotic behaviour. 

One of the main conclusions of this paper may be summarized as follows. While the von~Mises distribution has strong relations with the notion of extrinsic moments (that is with trigonometric moments), we will emphasize, in this paper, that the geodesic Normal distribution has strong relations with intrinsic moments that is with the Fréchet mean (defined as the angle $\alpha$ minimizing the expectation of $d_G(\alpha,\theta)^2$) and the geodesic variance. $vM$ and $gN$ distributions definition are closely related to respectively extrinsic and intrinsic moments, and we present their similarities together with dissimilarities.\\

After introducing the $gN$ distribution in Section \ref{sec-geoNorm}, we present a brief review on intrinsic and extrinsic quantities that allow characterization of distributions on the circle in Section \ref{sec-mes}. In Section \ref{sec-prop}, we present extrinsic and intrinsic properties of $vM$ and $gN$ distributions. Then, in Section \ref{sec-simu}, we rapidly explain how to simulate $gN$ random variables on the circle. Finally, in Section \ref{sec-mv}, we present the maximum likelihood estimators for the $gN$ distributions and study their asymptotic behaviour. Numerical simulations illustrate the presented results. 

\section{Geodesic Normal distribution through the tangent space}
\label{sec-geoNorm}
As introduced in \cite{A-Pen06} the geodesic Normal distribution is defined for random variables taking values on Riemannian manifolds and is based on the ``geodesic distance'' concept and on the total intrinsic variance \cite{A-BhaPat03,A-BhaPat05,A-Pen06}. On a Riemannian manifold ${\cal M}$, one can define at each point $x \in {\cal M}$ a scalar product $\left< .,. \right>_{x}$ in the tangent plane $T_x{\cal M}$ attached to the manifold at $x$. 
On ${\cal M}$, among the possible smooth curves between two points $x$ and $y$, the curve of minimum length is called a {\em geodesic}. The length of the curve is understood as integration of the norm of its instantaneous velocity along the path, and with the norm at position $x$ on ${\cal M}$ taken as:  $\left\|.\right\|=\sqrt{\left< .,. \right>_{x}}$. 
It is well-known that, given a point $x \in {\cal M}$ and a vector $\overrightarrow{v} \in T_x{\cal M}$, there exists only one geodesic $\gamma(t)$ with $\gamma(t=0)=x$ and with tangent vector $\overrightarrow{v}$.   

Through the exponential map, each vector $\in T_x{\cal M}$ is associated to a point $y \in {\cal M}$ reached in unit time,  {\em i.e.} $\gamma(1)=y$. Using the notation adopted in \cite{A-Pen06}, the vector defined in $T_x{\cal M}$ associated to the geodesic that starts from $x$ at time $t=0$ and reaches $y$ at time $t=1$ is denoted $\overrightarrow{xy}$. Thus, the exponential map (at point $x$) maps a vector of $T_xM$ to a point $y$, {\em i.e.} 
$y=\exp_x(\overrightarrow{xy})$.
Now the {\em geodesic distance}, denoted $d_{G}(x,y)$, between $x \in {\cal M}$ and $y \in {\cal M}$ is:
$$d_{G}(x,y)=\sqrt{\left<\overrightarrow{xy},\overrightarrow{xy}\right>_x}$$
The Log map is the inverse map that associates to a point $y \in {\cal M}$ in the neighbourhood of $x$ a vector $\overrightarrow{xy} \in T_x{\cal M}$, {\em i.e.} $\overrightarrow{xy}=Log_x(y)$.

A random variable $Y$ taking values in ${\cal M}$, with density function $f(y;\mu,\Gamma)$ is said to have a geodesic Normal distribution with parameters $\mu \in {\cal M}$ and $\Gamma$, a $(N,N)$ matrix, denoted by $gN(\mu,\Gamma)$, if:
$$
f(y;\mu,\Gamma)=k^{-1}\exp\left(-\frac{\overrightarrow{\mu y}^T.\Gamma.\overrightarrow{\mu y}}{2}\right).
$$ 
The parameter $\mu$ is related to some specific location concept. Namely, \cite{A-Pen06} has proved that $\mu$ corresponds to the intrinsic or Fréchet mean of the random variable $Y$ (see Section \ref{sec-mes} for more details).
The normalizing constant is given by:
$$
k=\int_{\cal M} \exp\left(-\frac{\overrightarrow{\mu y}^T.\Gamma.\overrightarrow{\mu y}}{2}\right) d{\cal M}(y)
$$
where $d{\cal M}(y)$ is the Riemannian measure (induced by the Riemannian metric). 
The matrix $\Gamma$ is called the concentration matrix and is related to the covariance matrix of the vector $\overrightarrow{\mu Y}$ given for random variables on a Riemannian manifold by
$$
\Sigma:=E[ \overrightarrow{\mu Y} \overrightarrow{\mu Y}^T]  =k\int_{\cal M} \overrightarrow{\mu y}.\overrightarrow{\mu y}^T\exp\left(-\frac{\overrightarrow{\mu y}^T.\Gamma.\overrightarrow{\mu y}}{2}\right) d{\cal M}(y).
$$
Note that in the case where ${\cal M}=\mathbb{R}^d$, the manifold is flat and the geodesic distance is nothing more than the Euclidian distance. In this case, 
we retrieve the classical definition of a Gaussian variable in $\R^d$ with $\mu$ and $\Gamma$ corresponding respectively to  the classical expectation and to the inverse of the covariance matrix $\Sigma^{-1}$.

\paragraph{The case of the circle} 
We now present, as done in \cite{A-Pen06}, the case where ${\cal M}$ is the unit circle ${\cal S}^1$. The exponential chart is the angle $\theta \in ]-\pi;\pi[$. Note that this chart is "local" and is thus defined at a point on the manifold. This must be kept in mind especially for explicit calculation. Here, as the tangent plane to ${\cal S}^1$ is the real line, $\theta$ takes values on the segment between $-\pi$ and $\pi$. Note that $\pi$ and $-\pi$ are omitted as they are in the cut locus of the "development point" (point on the manifold where the tangent plane is attached). $d{\cal M}$ is simply $d\theta$ here.

As stated before, the $gN$ distribution on the circle has density function  given for $\theta\in (\mu-\pi,\mu+\pi)$ by:
$$
f(\theta;\mu,\gamma) = k^{-1}(\gamma) e^{-\frac{\gamma}2 d_G(\mu,\theta)^2},
$$
where $\gamma$ is a nonnegative real number. Note that $d_G(\mu,\theta)$ is the arc length between $\mu$ and $\theta$. The normalization $k(\gamma)$ is:
$$k(\gamma)= \int_{\mu-\pi}^{\mu+\pi} e^{-\frac{\gamma}2 d_G(\mu,\theta)^2} d\theta = \sqrt{\frac{2\pi}\gamma} erf\left(\pi \sqrt{\frac\gamma 2} \right),$$ 
with the development made around $\mu$.

In order to consider a $gN$ distribution as a classical circular distribution (defined by axioms (i)-(iii) in the introduction), one must extend the support of this distribution from $(\mu-\pi,\mu+\pi)$ to $\R$. The way to achieve this is to make the geodesic distance periodic. Let us consider the distance $\tilde d_G$ for an angle $\alpha \in\R\setminus\{\mu +k\pi,k\in\mathbb{Z}\}$ defined by $\tilde d_G (\mu,\alpha)=d_G(\mu, \tilde\alpha)$ with $\tilde\alpha=(\alpha-\mu+\pi) (mod \; 2\pi)$. Let $\tilde f$ the density $f$ where $d_G$ is replaced by $\tilde{d}_G$. This new density defines a circular disribution satisfying axioms (i)-(iii). In particular,
$$
\int_0^{2\pi} \tilde{f} (\theta,\mu,\gamma)d\theta = \int_0^{2\pi} k^{-1}(\gamma)e^{-\frac{\gamma}2 \tilde d_G(\mu,\theta)^2}d\theta= \int_{\mu-\pi}^{\mu+\pi} k^{-1}(\gamma)e^{-\frac{\gamma}2  d_G(\mu,\theta)^2}d\theta=1.
$$
For the sake of simplicity, $d_G$ will be understood as the distance $\tilde d_G$ in the rest of the paper. And therefore, the density of a $gN$ distribution is considered periodic, defined in $\R\setminus\{\mu +k\pi,k\in\mathbb{Z}\}$ and with values on $\R$. In the case of a $vM$ distribution (and actually for most of circular distributions) no such considerations are needed; the periodic nature being included through the $\cos(\cdot)$ function.

\section{Classical measures of location and dispersion for circular random variables}\label{sec-mes}

We briefly present the concepts of {\em extrinsic} and {\em intrinsic} moments for random variables on the circle. While the former are well-known in circular statistics ({\it e.g.} \cite{B-MarJup00}), the later, based on the geodesic distance on the circle, are less used in this domain. They have been introduced and commonly used when dealing with genereal Riemannian manifolds ({\it e.g.} \cite{A-Kar77,A-Zie77,A-BhaPat03} and the numerous references therein).

\subsubsection*{Extrinsic moments}

In circular statistics, it is well established that trigonometric moments give access to measures of {\em mean direction} and {\em circular variances}. Considering a cicular random variable $\theta$, its $p-{th}$ order trigonometric moment is defined as $\varphi_p=E[e^{ip\theta}]=\alpha_p + i \beta_p$ where $\alpha_p=E[\cos p\theta]$ and $\beta_p=E[\sin p\theta]$. These later quantities are {\em extrinsic} by definition. The first order trigonometric moment is thus $\varphi_1=E[e^{i\theta}]=\rho e^{i \mu^E}$
where $\rho$ is called the {\em mean resultant length} ($0\leq \rho \leq 1$) and $\mu^E$ is the {\em mean direction}. In the following, we refer to $\mu^E $ as the {\em extrinsic mean}. The {\em extrinsic variance} $\sigma^2_E$ is indeed the circular variance defined as $\sigma^2_E = 1 - \rho$. In the sequel, {\em extrinsic moments} will be used in place of {\it trigonometric moments}, keeping in mind that they are the same quantities. For more details on trigonometric moments see \cite{B-MarJup00,B-JamSen01}.

\subsubsection*{Intrinsic moments}
Another way to consider moments for distributions of random variables on the circle is to use the fact that ${\cal S}^1$ is a Riemannian manifold and thus the geodesic distance can be used to define {\em intrinsic moments}. Given a random variable $\theta$ with values on ${\cal S}^1$, we define by $\mu^ I=\mbox{argmin}_{\tilde\mu \in \mathcal{S}^1} E[ d_G(\tilde\mu,\theta)^2]$,
the intrinsic mean set (a particular case of a Fréchet mean set), where we recall that $d_G(.,.)$ is the geodesic distance on the circle, {\em i.e.} the arc length. When the intrinsic mean set is reduced to a single element, $\mu^I$ is simply called the intrinsic mean.
The {\em intrinsic variance} $\sigma^2_I$ is then uniquely defined by $\sigma^ 2_I=E\left[d_G(\mu^I,\theta)^2\right]$ where $\mu^I$ is the intrinsic mean (set) defined above.
For more details and a thourough study of intrinsic statistics for random variables on Riemannian manifolds, see \cite{A-Pen06} or \cite{A-BhaPat03,A-BhaPat05}. Other concepts of intrinsic variance (e.g. variances obtained by residuals and by projection) exist (see \cite{A-Huc01} for thourough description). In this paper, we only focus on the intrinsic mean and (total) variance which are sufficient to basically compare $gN$ and $vM$ distributions.

\section{Basic properties of the $gN(\mu,\gamma)$ and $vM(\mu,\kappa)$ distributions}
\label{sec-prop}
\subsubsection*{Symmetry property}

First, let us say that like the $vM$ distribution, the $gN$ distribution has a mode for $\theta=\mu$, and a symmetry around $\theta=\mu$. The $vM$ distribution has an anti-mode for $\theta=\mu \pm \pi$. For a $gN$ distribution, the density is not defined at these points. However, the shared behaviour is the decreasing of both densities on each interval $(\mu-\pi,\mu]$ and $[\mu,\mu+\pi)$.

\subsubsection*{Extrinsic and intrinsic means and variances}

Table~\ref{sum-mes} summarizes extrinsic and intrinsic means and variances for both distributions of interest. Let us make some comments.
In~\cite{A-KazSri08}, the authors follow the works of Le (\cite{A-Le98,A-Le01}) and give very simple conditions on the density of a circular random variable that ensure the existence and unicity of the intrinsic mean. It is left to the reader that applying Theorem~1 of \cite{A-KazSri08} allows us to assert that the intrinsic mean of a $gN(\mu,\gamma)$ or $vM(\mu,\kappa)$ is $\mu$. Furthermore, the computation of $\sigma^2_I$ for a $gN$ distribution (resp. $\sigma^2_E$ for a $vM$ distribution) can be found in \cite{A-Pen06} (resp. {\it e.g.} \cite{B-MarJup00}). The intrinsic variance $\sigma_I^2$ for a $vM$ is quite obvious and is omitted. It remains to explain how we obtain the extrinsic moments for a $gN$ distribution. Both are indeed derived through the $p$th trigonometric moment of a $gN$ distribution reported in the following proposition.

\begin{proposition}\label{prop-TM}
The $p$-th trigonometric moment ($p \in \N^*$) of a $gN(\mu,\gamma)$ distribution, denoted by $\varphi_p$ and defined by $\varphi_p := E\left[ e^{ip\theta} \right]$ is given by
$$
\varphi_p  = e^{ip\mu} \;e^{-\frac{p^2}{2\gamma}}\; \frac{ Re \left( erf\left( \pi \sqrt{\frac\gamma 2}- i \frac{p}{\sqrt{2\gamma}} \right)\right)}{erf\left( \pi \sqrt{\frac\gamma 2}\right)},
$$
where $erf$ is the error function defined for any complex number by $erf(z)=\frac{2}{\sqrt{\pi}} \int_{-\infty}^z e^{-t^2} dt$.
\end{proposition}

\begin{proof}
Let $p\geq 1$, 
$$
E[ e^{ip\theta}] = k^{-1}(\gamma) \int_0^{2\pi} e^{ip\theta} e^{-\frac\gamma 2 d_G(\theta,\mu)^2} d\theta = k^{-1}(\gamma) \int_{\mu-\pi}^{\mu+\pi} e^{ip\theta} e^{-\frac\gamma 2 d_G(\theta,\mu)^2} d\theta.
$$
Since for $\theta\in (\mu-\pi,\mu+\pi)$, $d_G(\theta,\mu)=|\theta-\mu|$, 
\begin{eqnarray*}
E[ e^{ip\theta}] &=& k^{-1}(\gamma) \int_{-\pi}^\pi e^{ip(\theta+\mu)} e^{-\frac{\gamma}2 \theta^2} d\theta  \\
&=& e^{ip \mu} k^{-1}(\gamma) \int_{-\pi}^\pi e^{- \left(  \left(\theta\sqrt{\frac{\gamma}2} - i \frac p{\sqrt{2\gamma}}\right)^2 - \left( i \frac p{\sqrt{2\gamma}}\right)^2 \right)}d\theta  \\
&=& e^{ip\mu} e^{-\frac{p^2}{2\gamma}} k^{-1}(\gamma) \sqrt{\frac{2\pi}\gamma } \frac{1}2 \left( erf\left( \pi\sqrt{\frac{\gamma}2}- i \frac{p}{\sqrt{2\gamma}} \right) - erf\left( -\pi\sqrt{\frac{\gamma}2}- i \frac{p}{\sqrt{2\gamma}} \right) \right) \\
&=& e^{ip\mu} \; e^{-\frac{p^2}{2\gamma}}\; \frac{ Re \left( erf\left( \pi \sqrt{\frac\gamma 2}- i \frac{p}{\sqrt{2\gamma}} \right)\right)}{erf\left( \pi \sqrt{\frac\gamma 2}\right)},
\end{eqnarray*}
since for any complex number $z$, $erf(z)-erf(-\overline{z})=erf(z)-\overline{erf(-z)}=erf(z)+\overline{erf(z)}=2 Re(erf(z))$.
\end{proof}

\begin{table}[H]
\begin{center}\begin{tabular}{lcccc}
\hline
Distribution& $\mu^I$ & $\mu^E$& $\sigma_I^2$ & $\sigma_E^2$ \\
\hline
$vM(\mu,\kappa)$  & $\mu$ &$\mu$ & $\frac{1}{2\pi I_0(\kappa)} \int_{-\pi}^\pi \alpha^2 e^{\kappa \cos(\alpha)}d\alpha$ & $1-\frac{I_1(\kappa)}{I_0(\kappa)}$ \\
$gN(\mu,\gamma)$& $\mu$ &$\mu$ &$\frac1\gamma\left( 1- 2\pi k^{-1}(\gamma) e^{-\frac{\gamma \pi^2}2} \right)$  & $1-e^{-\frac{1}{2\gamma}}\frac{ Re \left( erf\left( \pi \sqrt{\frac\gamma 2}- i \frac{1}{\sqrt{2\gamma}} \right)\right)}{erf\left( \pi \sqrt{\frac\gamma 2}\right)}$ \\
\hline
\end{tabular}
\end{center}
\caption{Summary of extrinsic and intrinsic means and variances for the von~Mises and geodesic Normal distributions. The function $I_p(\cdot)$ denotes the modified Bessel function of the first kind of order $p$.}\label{sum-mes}
\end{table}

Figure~\ref{fig-conc} shows the evolutions of the extrinsic and intrinsic variances in terms of the concentration parameter $\kappa$ for the $vM$ and $\gamma$ for the gN. It is interesting to notice that the von~Mises distribution and the geodesic Normal distributions have intrinsic variance equal to $\frac{\pi^2}3$ when $\gamma$ or $\kappa$ equals zero, corresponding to the variance of the uniform distribution on the circle. Note also that intrinsic and extrinsic variances tend to zero as $\gamma$ or $\kappa$ tend to infinity.

\begin{figure}[H]
\centerline{\includegraphics[scale=.6]{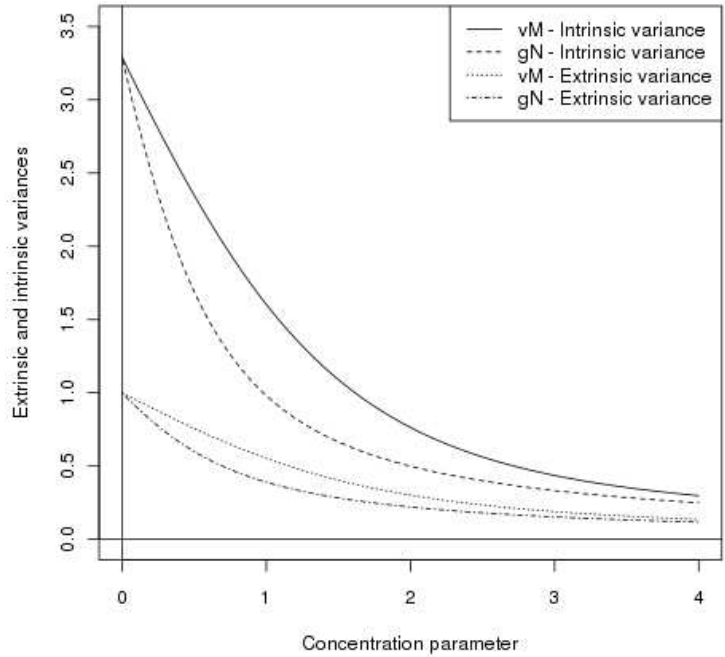}}
\label{fig-conc}
\caption{Evolutions of extrinsic and intrinsic variances in terms of the concentration parameter, $\gamma$ for the $gN$ distribution and $\kappa$ for the $vM$ one.}
\end{figure}

\subsubsection*{Entropy property} \label{entropy}

$gN$ and $vM$ have both a peculiar position respectively amongst distributions defined on Riemannian manifolds and circular statistics distributions: both maximize a certain definition of the entropy. As explained in \cite{B-JamSen01} (characterization due to Mardia \cite{B-Mar72}),  the circular distribution that maximizes the entropy defined with respect to the angular random variable $\theta$ (see \cite{B-MarJup00} for exact definition), subject to the constraint that the first trigonometric moment is fixed, {\em i.e.} for $\mu^E$ and $\sigma^2_E$ fixed, is a vM($\mu^E$,$\kappa$), where $\sigma_E^2=1-I_1(\kappa)/I_0(\kappa)$. In a similar way, as demonstrated in \cite{A-Pen06}, the distribution defined using the geodesic distance on ${\cal S}^1$ which maximizes the entropy, when it is defined in the tangent plane, and subject to the constraints that $\mu^I$ and $\sigma^2_I$ are fixed, is the $gN(\mu^I$,$\sigma^2_I)$. One can thus conclude that $vM$ and $gN$ distributions play ``similar'' roles in that they maximize the entropy with respect to either extrinsic or intrinsic moments of the distribution.


\subsubsection*{Linear approximation of  $\overline{\mu\theta}$}

Recall that the random variable $\overline{\mu\theta}$ represents the algebraic measure of the vector $\overrightarrow{\mu \theta}$. The support of $\overline{\mu\theta}$ is $(-\pi,\pi)$. Its cdf is given for $t \in (-\pi,\pi)$ by
$$
F_{\overline{\mu\theta}}(t) = k^{-1}(\gamma)\int_{\mathcal{S}^1} \mathbf{1}_{[-\pi,t]}(\overline{\mu y}) e^{-\frac\gamma 2 \overline{\mu y}^2} d\mathcal{M}(y) = k^{-1}(\gamma) \int_{-\pi}^t e^{-\frac{\gamma}2 \theta^2} d\theta.
$$
This reduces to $F_{\overline{\mu\theta}}(t) = \frac{ \Phi\left( t \sqrt{\gamma} \right)  - \Phi\left( -\pi \sqrt{\gamma} \right)}{\Phi\left( \pi \sqrt{\gamma} \right)  - \Phi\left( -\pi \sqrt{\gamma} \right)}$. In other words, $\overline{\mu\theta}$ is nothing else than a truncated Gaussian random variable with support $(-\pi,\pi)$ with mean 0 and scale parameter $1/\sqrt{\gamma}$, that is
\begin{equation}\label{eq-trunc}
\overline{\mu \theta} \; \stackrel{d}{=} \; Z \big | \; |Z|\leq \pi, \mbox{ where } Z \sim \mathcal{N}(0, 1/\sqrt{\gamma}).
\end{equation}

For large concentration parameters $\gamma$, the geodesic Normal distribution can be ``approximated'' by a linear normal distribution in the following sense.

\begin{proposition}
Let $\theta \sim gN(\mu,\gamma)$, then as $\gamma\to +\infty$, $\sqrt{\gamma} \; \overline{\mu\theta} \stackrel{d}{\to} \mathcal{N}(0,1)$.
\end{proposition}

\begin{proof}
Let us denote by $h (\cdot)$ the moment generating function of the random variable $\sqrt{\gamma} \; \overline{\mu\theta}$, then for $t\in \mathbb{R}$
\begin{eqnarray*}
h (t) &=& k^{-1}(\gamma) \int_{-\pi}^\pi e^{\sqrt{\gamma}t \theta} e^{-\frac\gamma 2 \theta^2}d\theta  \\
&=& e^{\frac{t^2}2} \; k^{-1}(\gamma) \sqrt{\gamma} \int_{-\pi\sqrt{\gamma}}^{\pi\sqrt{\gamma}} e^{-\frac12 (\theta-t)^2}d\theta \\
&=& e^{\frac{t^2}2} \; \frac{ \Phi\left( \pi \sqrt{\gamma}-t \right)  - \Phi\left( -\pi \sqrt{\gamma}-t \right)}{\Phi\left( \pi \sqrt{\gamma} \right)  - \Phi\left( -\pi \sqrt{\gamma} \right)}.
\end{eqnarray*}
Therefore, as $\gamma \to +\infty$, for fixed $t$, $h(t)$ converges towards $e^{t^2/2}$ which is the moment generating function of a standard Gaussian random variable.
\end{proof}

Such a result also holds for the von~Mises distribution: as $\kappa\to +\infty$, $\sqrt\kappa (\theta-\mu)\stackrel{d}{\to}\mathcal{N}(0,1)$, see {\it e.g.} Proposition~2.2 in \cite{B-JamSen01}.

\section{Simulation of a geodesic Normal distribution and examples}
\label{sec-simu}
The generation of a $gN$ distribution with support on $(0,2\pi)$ is extremely simple following (\ref{eq-trunc}). It consists in two steps.
\begin{enumerate}
\item Generate $Z \big | \; |Z|\leq \pi, \mbox{ where } Z \sim \mathcal{N}(0, 1/\sqrt{\gamma})$
\item Set $\theta= \mu+ Z (\mbox{mod } 2\pi)$.
\end{enumerate}

Figure~\ref{fig-exgN} presents some examples for different values of location parameter $\mu$ and concentration parameter $\gamma$.

\begin{figure}[htbp]
\begin{center}
\begin{tabular}{ll}
\includegraphics[scale=.37]{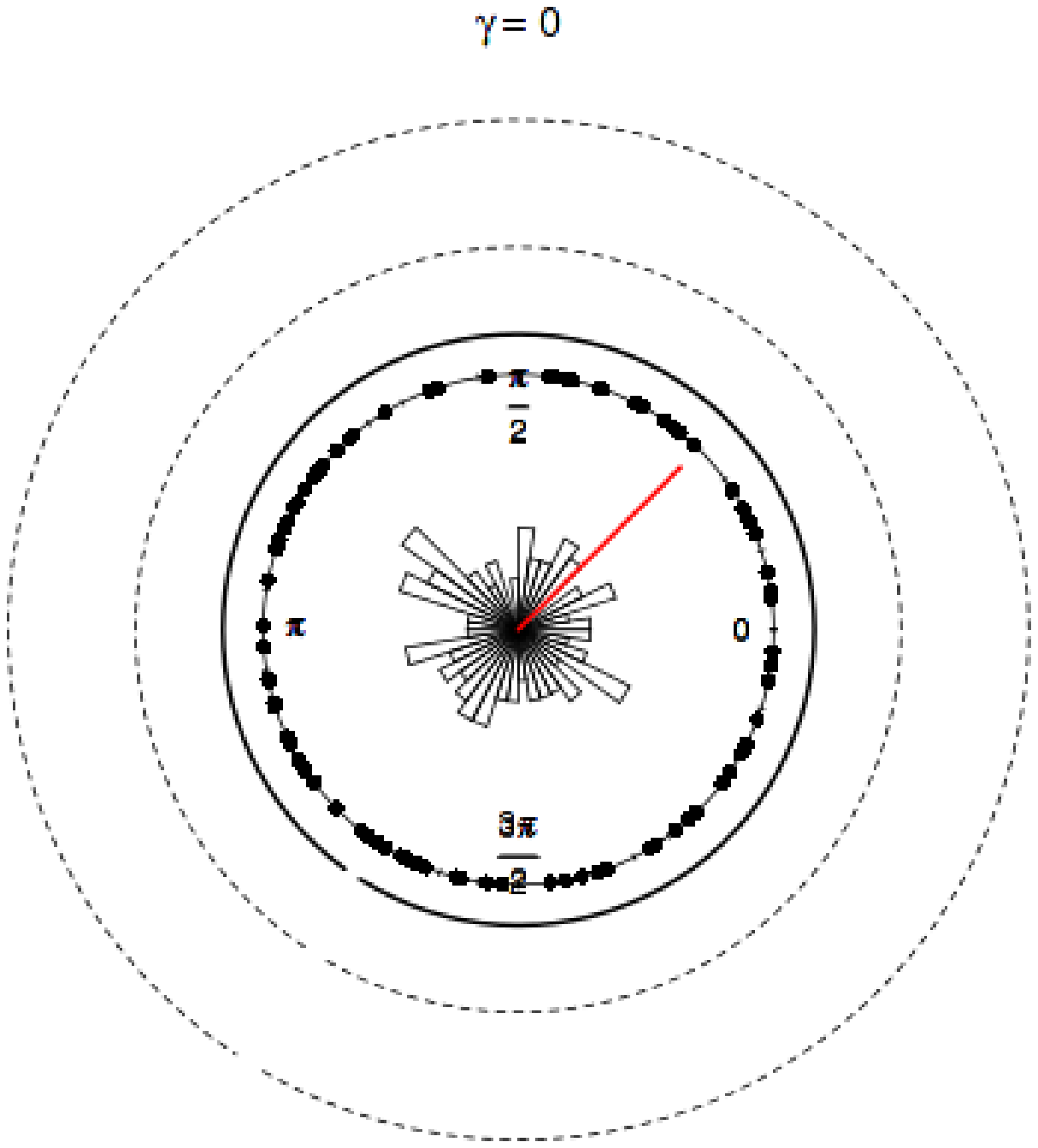} & \includegraphics[scale=.37]{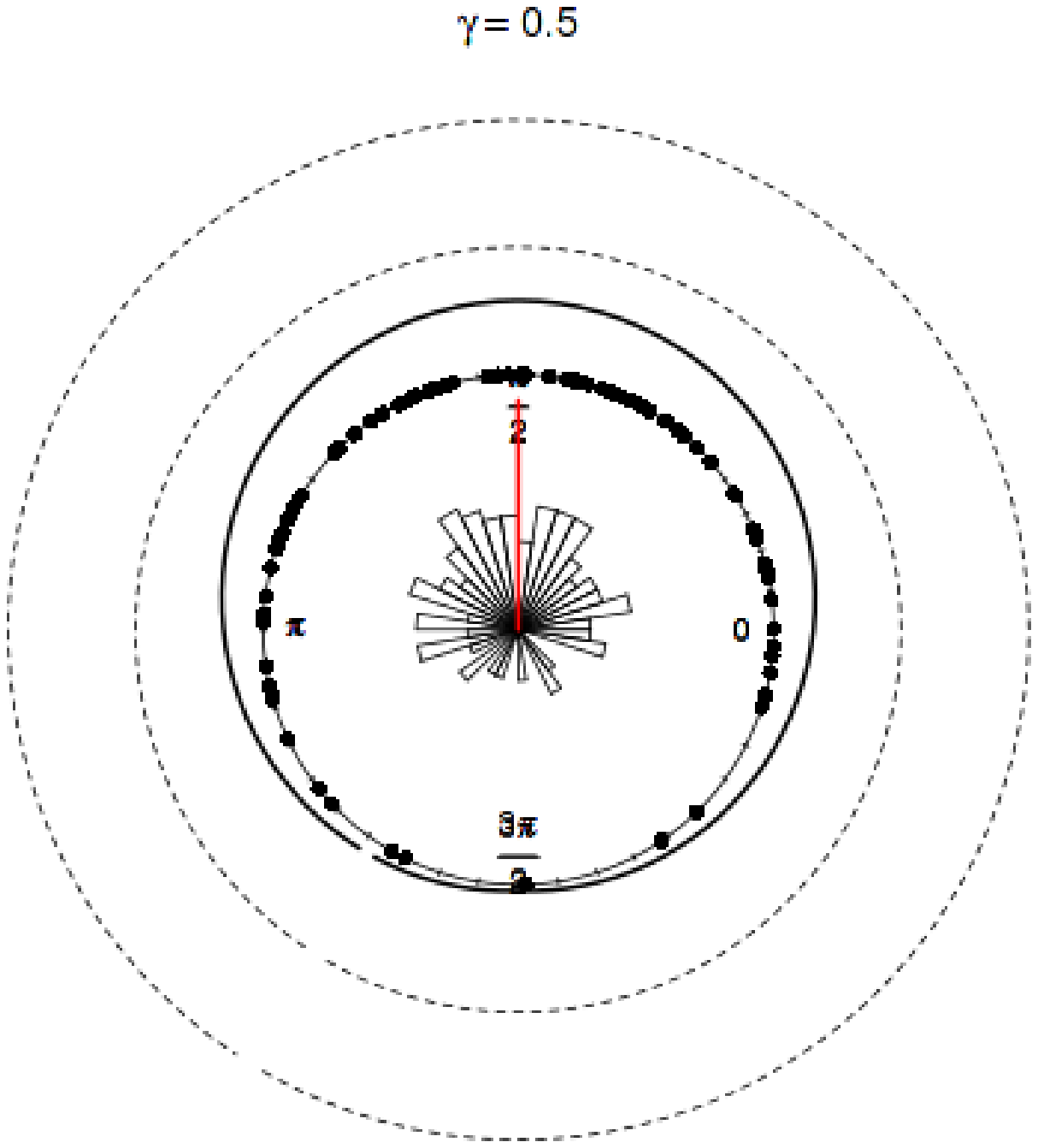} \\
\includegraphics[scale=.37]{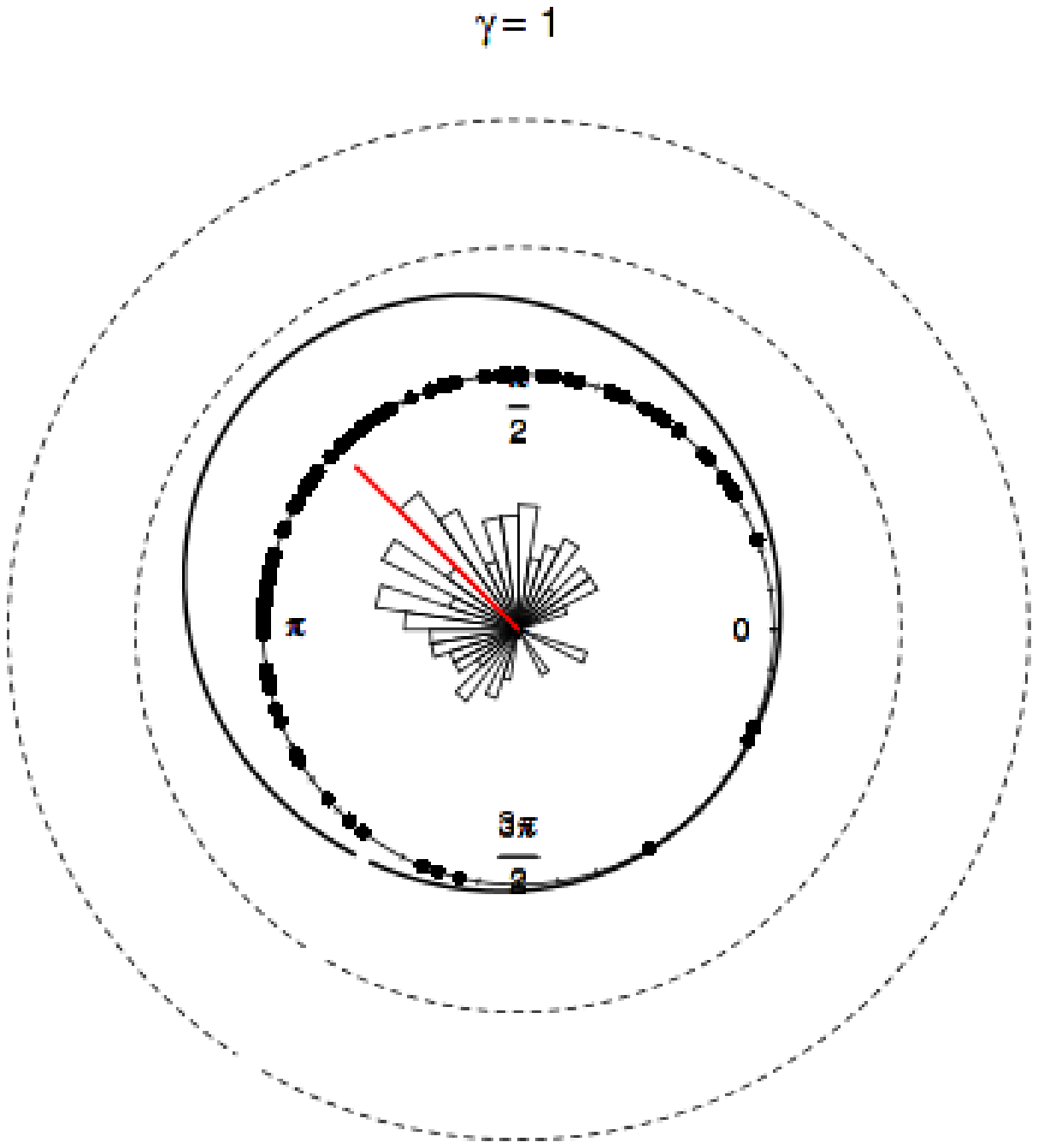} & \includegraphics[scale=.37]{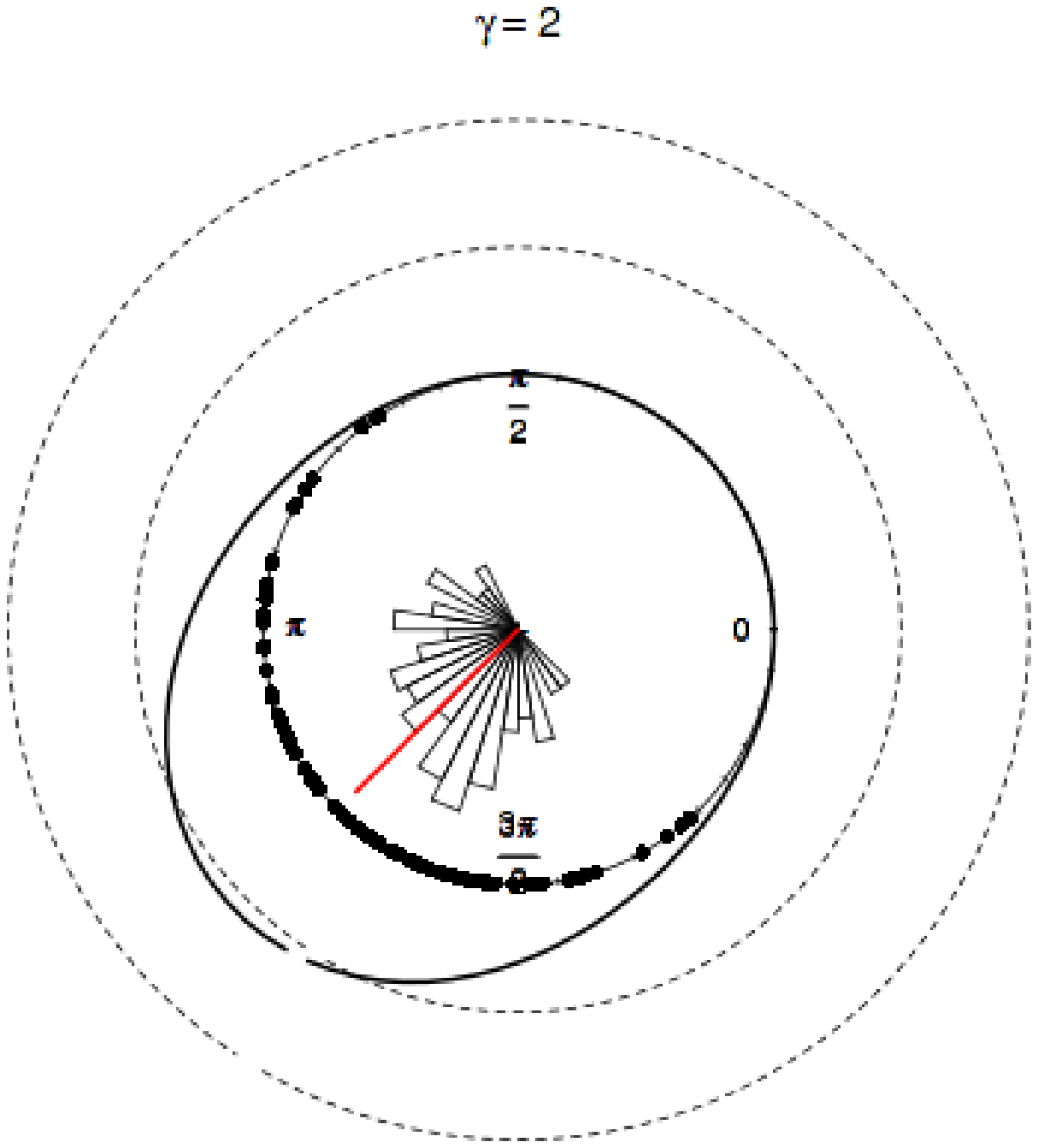}\\
\includegraphics[scale=.37]{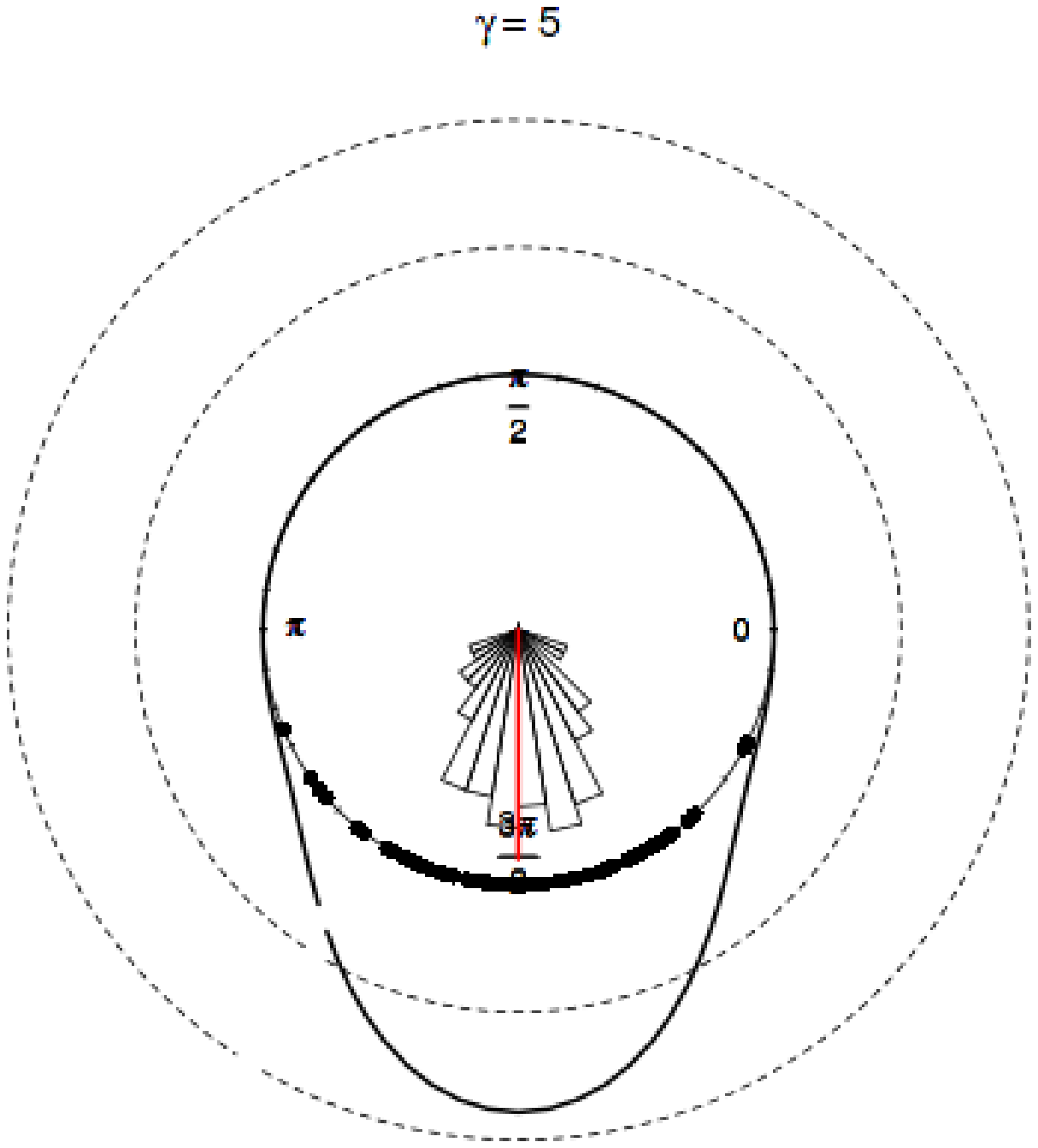} & \includegraphics[scale=.37]{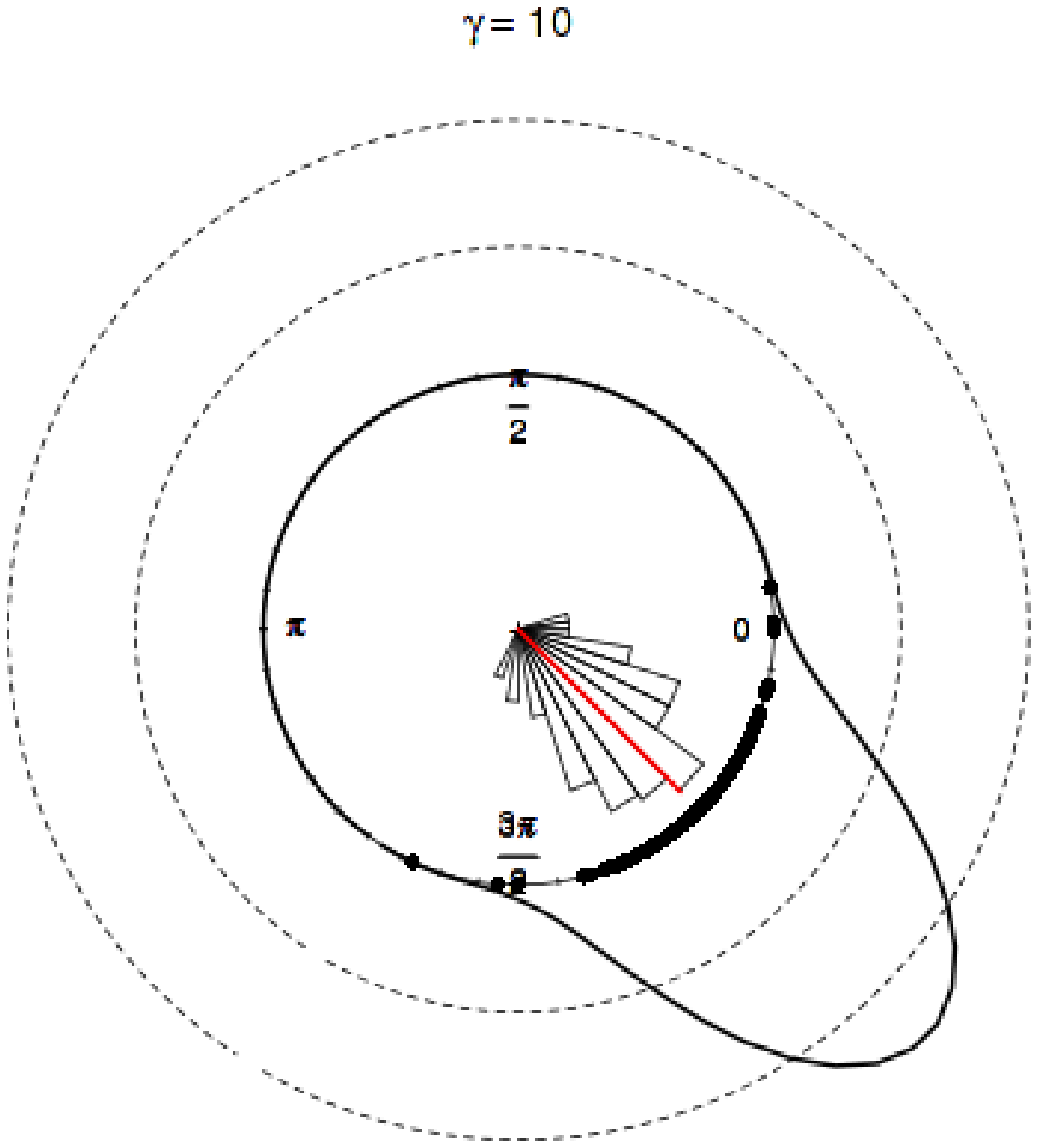}
\end{tabular}
\caption{Example of $n=100$ $gN(\mu,\gamma)$ realizations for different parameters. The red line indicates the value of $\mu$. These different plots have been produced using the \texttt{R} packages \texttt{circular, CircStats} maintained by C. Agostinelli, and related to the book \cite{B-JamSen01}.}\label{fig-exgN}
\end{center}

\end{figure}

\section{Maximum Likelihood Estimation}
\label{sec-mv}

\subsection{Preliminary and notation}

Let us consider now the identification problem of estimating the parameters of a $gN$ distribution from the $n$ observations $\theta_1,\ldots,\theta_n \in (0,2\pi)$. In this section, we will denote by $\mu^\star$ and $\gamma^\star$ the unknown parameters to estimate. We assume that $\mu^\star \in [0,2\pi)$ and $\gamma^\star>0$.
Also, we propose to denote by $\widehat{\mu}^I$ and $\widehat{\mu}^E$ the empirical intrinsic  and extrinsic means defined by
\begin{eqnarray} 
\widehat{\mu}^I &:=& \mbox{argmin}_{\mu \in \mathcal{S}^1} \frac1n \sum_{i=1}^n d_G(\mu,\theta_i)^2.  \label{def-muIest}\\
\widehat{\mu}^E &:=& \mbox{Arg}\left(\widehat{\varphi}_1 \right), \mbox{ with } \widehat{\varphi}_1:=\frac1n \sum_{j} \cos(\theta_j)+ i \frac{1}n \sum_j \sin(\theta_j). \label{def-muEest}
\end{eqnarray} 
Obtained through the minimization of an empirical function, $\widehat{\mu}^I$ is not necessarily reduced to a single element. The natural intrinsic and extrinsic variances are then denoted by $\widehat{\sigma}_I^2$ and $\widehat{\sigma}_E^2$ and uniquely given by
\begin{equation} \label{def-sigIEest}
\widehat{\sigma}_I^2 =  \frac1n \sum_{i=1}^n d_G(\widehat{\mu}^I,\theta_i)^2
 \quad \mbox{ and } \quad 
\widehat{\sigma}_E^2 =  1- |\widehat{\varphi}_1|.
\end{equation}
In the following, we will need the following Lemma and notation.
\begin{lemma} \label{lem-sig}
For a random variable $\theta_\gamma \sim gN(\mu^\star,\gamma)$, let $V(\mu,\gamma):= E[ d_G(\mu , \theta_\gamma)^2]$ for $\mu \in (\mu^\star-\pi, \mu^\star+\pi)$ and $\delta=\mu^\star-\mu$, then
$$
V(\mu,\gamma) = g(\delta) \mathbf{1}_{[0,\pi)}(\delta) +g(-\delta) \mathbf{1}_{(-\pi,0]}(\delta),
$$
where
$$g(\delta):=\int_{-\pi}^\pi (\delta+\alpha)^2 f(\alpha)d\alpha + 4\pi\int_{\pi-\delta}^\pi (2\pi - (\delta+\alpha))f(\alpha)d\alpha$$ 
and $f(\alpha)=k^{-1}(\gamma)e^{-\frac{\gamma}2\alpha^2}$.
\end{lemma}

\begin{proof}
Let us fix $\mu\in (\mu^\star-\pi,\mu^\star]$, then
$$
\overline{\mu \alpha}= \left\{  \begin{array}{ll}
\overline{\mu\mu^\star}+ \overline{\mu^\star\alpha} & \mbox{when } \alpha \in (\mu^\star-\pi,\mu^\star+\pi-(\mu^\star-\mu)) \\
\overline{\mu\mu^\star}+ \overline{\mu^\star\alpha}-2\pi & \mbox{when } \alpha \in (\mu^\star+\pi-(\mu^\star-\mu),\mu^\star+\pi)
\end{array} \right.
$$
Denoting $\delta:=\mu^\star-\mu$, this expansion allows us to derive
\begin{eqnarray*}
V(\mu,\gamma)&=& E[d_G(\mu,\theta_\gamma)^2] \\
&=&  \int_{-\pi}^{\pi-\delta} (\delta+\alpha)^2 f(\alpha)d\alpha +\int_{\pi-\delta}^\pi (2\pi-\delta-\alpha)^2 f(\alpha)d\alpha \\
&=& \int_{-\pi}^\pi (\delta+\alpha)^2 f(\alpha)d\alpha + 4\pi\int_{\pi-\delta}^\pi (2\pi - (\delta+\alpha))f(\alpha)d\alpha.
\end{eqnarray*}
Now, let $\mu\in [\mu^\star,\mu^\star+\pi)$, then
$$
\overline{\mu \alpha}= \left\{  \begin{array}{ll}
\overline{\mu\mu^\star}+ \overline{\mu^\star\alpha} & \mbox{when } \alpha \in (\mu^\star-\pi+(\mu-\mu^\star),\mu^\star+\pi) \\
\overline{\mu\mu^\star}+ \overline{\mu^\star\alpha}+2\pi & \mbox{when } \alpha \in (\mu^\star-\pi,\mu^\star-\pi+(\mu-\mu^\star)),
\end{array} \right.
$$
which leads to
\begin{eqnarray*}
V(\mu,\gamma)&=&\int_{-\pi-\delta}^{\pi} (\delta+\alpha)^2 f(\alpha)d\alpha +\int_{-\pi}^{-\pi-\delta} (2\pi+\delta+\alpha)^2 f(\alpha)d\alpha \\
&=& \int_{-\pi}^\pi (\delta+\alpha)^2 f(\alpha)d\alpha + 4\pi\int_{-\pi}^{-\pi-\delta} (2\pi + (\delta+\alpha))f(\alpha)d\alpha \\
&=&\int_{-\pi}^\pi (\delta+\alpha)^2 f(\alpha)d\alpha + 4\pi\int_{\pi+\delta}^{\pi}(2\pi + (\delta-\alpha))f(\alpha)d\alpha \\
&=& g(-\delta).
\end{eqnarray*}
\end{proof}

Obviously $V(\mu^\star,\gamma)$ corresponds to the intrinsic variance of $\theta_\gamma$. Then, from Table~\ref{sum-mes}, this function does not depend on $\mu^\star$ and will therefore be simplified to $V(\gamma)$. As used in Theorem~\ref{prop-mle}, let us recall here the expression of the later quantity.
\begin{equation} \label{def-s2Igamma}
V(\gamma)=V(\mu^\star,\gamma):=\frac1\gamma\left( 1- 2\pi k^{-1}(\gamma) e^{-\frac{\gamma \pi^2}2} \right) 
\quad \mbox{ with } k(\gamma)=\sqrt{\frac{2\pi}\gamma} erf\left(\pi \sqrt{\frac\gamma 2} \right)
\end{equation}

\subsection{Maximum Likelihood Estimate}

The log-likelihood expressed for $n$ i.i.d. $gN$ distributions is given by~:
$$
\ell(\mu,\gamma) = - n \log(k(\gamma)) - \frac\gamma 2 \sum_{i=1}^n d_G(\mu , \theta_i )^2.
$$
Let $\Omega=\left\{ (\mu,\gamma):\mu \in [0,2\pi)\setminus \{ (\mu^\star \pm \pi) \; mod(2\pi)\}, \gamma>0\right\}$ and assume that the true parameter $(\mu^\star,\gamma^\star)$ belongs to the interior of $\Omega$. The MLE estimates and asymptotic results are given by the following result. 

\begin{theorem}\label{prop-mle}${ }$\\
$(i)$ The MLE estimate of $\mu^\star$ corresponds to the intrinsic sample mean set, that is $\widehat{\mu}^{MLE}:=\widehat{\mu}^I$. The MLE estimate of $\gamma^\star$ is uniquely given by $\widehat{\gamma}^{MLE}=V^{-1}(\widehat{\sigma}_I^2)$, where $V(\cdot)$ is the function defined by (\ref{def-s2Igamma}) and where $\widehat{\sigma}_I^2$ is the intrinic sample variance.\\
$(ii)$ As $n\to +\infty$, $(\widehat{\mu}^{MLE},\widehat{\gamma}^{MLE})$ is a strongly consistent estimate of $(\mu^\star,\gamma^\star)$.\\
$(iii)$ As $n\to +\infty$, the MLE estimates satisfy the following central limit theorem
$$
\sqrt{n} \left( \widehat{\mu}^{MLE}-\mu^\star, \widehat{\gamma}^{MLE}-\gamma^{\star} \right)^T \stackrel{d}{\to} \mathcal{N}\left(0,J^{-1}(\gamma^\star)\right),$$
where $J(\gamma^\star)$ is the Fisher information matrix given by $J(\gamma^\star)=  
\left(\begin{array}{ll} J_1(\gamma^\star) & 0 \\ 0 & J_2(\gamma^\star) \end{array}
\right)$
with $J_1(\gamma^\star) := \gamma^\star \left(1 - 2\pi k^{-1}(\gamma^\star)e^{-\frac{\gamma^\star}2 \pi^2} \right)$ and 
$J_2(\gamma^\star) := \frac{k^{\prime\prime}(\gamma^\star)}{k(\gamma^\star)}- \left(\frac{k^\prime(\gamma^\star)}{k(\gamma^\star)}\right)^2. $
\end{theorem}

We emphasize that we do not derive analytic formulas for $\widehat{\gamma}^{MLE}$ but we prove its uniqueness by proving that the function $V(\cdot)$ is a strictly decreasing function (which is illustrated by Figure~\ref{fig-conc}). From a practical point of view, the computation of $\widehat{\gamma}^{MLE}$ (as well as the one of $\widehat{\mu}^{MLE}$) has been derived using a simple optimization algorithm.

As Mardia and Jupp did for the von~Mises distribution (\cite{B-MarJup00}, Section 5.3 p. 86), $\widehat{\mu}^{MLE}$ is regarded as unwrapped onto the line for the asymptotic normality result.

As it is for a Gaussian distribution on the real line or for a $vM$ distribution on the circle, the Fisher information matrix of a $gN$ distribution does not depend on the true location parameter $\mu^\star$ and the two estimates of $\mu^\star$ and $\gamma^\star$ are asymptotically independent. Let us also note that the  geodesic moment estimates, that is the estimates of $\mu^\star$ and $\gamma^\star$ based on the first two geodesic moments equations $\mu^I$ and $V(\gamma)$ exactly fit to the maximum likelihood estimates. Here is again another analogy with the $vM$ distribution, since the MLE of a $vM$ distribution correspond to the estimates of $\mu$ and $\kappa$ based on the extrinsic moments (see \cite{B-MarJup00} for further details).

\begin{proof}
$(i)$ Since the minimum of $\sum_{i=1}^n d_G(\mu , \theta_i )^2$ defines the intrinsic sample mean set, the MLE of $\mu$ correponds to the intrinsic sample mean set of $\mu^\star$. Now, the partial derivative of $\ell$ with respect to $\gamma$ is  given by
$$
\frac{\partial \ell}{\partial \gamma } (\mu,\gamma) = - \frac{k^\prime(\gamma)}{k(\gamma)} - \frac12 \sum_{i=1}^n d_G(\mu , \theta_i )^2.
$$
Let us note that 
\begin{equation}\label{kprime}
{k^\prime(\gamma)} =-\frac12 \int_{-\pi}^\pi \theta^2 e^{-\frac{\gamma}2 \theta^2} d\theta =-\frac{k(\gamma) }2 E[ d_G(\mu^\star,\theta_\gamma)^2] = - \frac{k(\gamma) }2 V(\gamma).
\end{equation}
Replacing $\mu$ by its MLE estimate and taking the derivative of $\ell$ w.r.t. $\gamma$ equal to zero implies that the MLE estimate of $\gamma$ is defined by the following equation:
$$
V(\widehat{\gamma}^{MLE}) = \frac1n \sum_{i=1}^n d_G(\widehat{\mu}^{MLE},\theta_i)^2 =: \widehat{\sigma}_I^2.
$$
The proof is ended by showing that $V(\cdot)$ is a striclty decreasing function on $\R^+$. Similarly to~(\ref{kprime}), we notice that $k^{\prime\prime}(\gamma)= \frac{k(\gamma)}4 E[d_G(\mu^\star,\theta_\gamma)^4]$. Now,
\begin{eqnarray}
V^\prime(\gamma) &=& -2 \left( \frac{k^{\prime\prime}(\gamma)}{k(\gamma)} - \left(\frac{k^\prime(\gamma)}{k(\gamma)}\right)^2 \right) \nonumber \\
&=& -2 \left( \frac14 E[ d_G(\mu^\star,\theta_\gamma)^4] - E[d_G(\mu^\star,\theta_\gamma)^2/2]^2 \right) \nonumber\\
&=& -\frac12 Var[ d_G(\mu^\star,\theta_\gamma)^2 ] <0. \label{dg4}
\end{eqnarray}
$(ii)$ As $\mu^\star$ corresponds to the unique intrinsic mean for a $gN$ distribution, the strong consistency of the intrinsic sample mean set is derived from {\it e.g.} Theorem~2.3 of \cite{A-BhaPat03}. Now, for the consistency of $\widehat{\gamma}^{MLE}$, let us introduce the variable $M_n:=n^{-1} \sum_{i=1}^n d_G(\mu^\star,\theta_i)^2$. From the LLN, $M_n$ converges almost surely towards $E[d_G(\mu^\star,\theta_i)^2]=V(\gamma^\star)$. Moreover, we may prove that $M_n-\widehat{\sigma}_I^2$ tends to zero (almost surely). The later following from
\begin{eqnarray*}
|M_n-\widehat{\sigma}_I^2| &=& \left| \frac1n \sum_{i=1}^n (d_G(\widehat{\mu}^{MLE},\theta_i)^2 - d_G(\mu^\star,\theta_i)^2 ) \right| \\
&\leq& \frac{4\pi}n \sum_{i=1}^n \left| d_G(\theta_i,\widehat{\mu}^{MLE}) - d_G(\theta_i,\mu^\star) \right| \\
&\leq&  \frac{4\pi}n \sum_{i=1}^n d_G(\widehat{\mu}^{MLE},\mu^\star) = 4\pi d_G(\widehat{\mu}^{MLE},\mu^\star).
\end{eqnarray*}
Combining the previous convergences leads to the almost sure convergence of $\widehat{\sigma}_I^2$ to $V(\gamma^\star)$ and to the result since $V(\cdot)$ is continuous and invertible on $\mathbb{R}_+^*$.

\noindent $(iii)$ Standard theory of maximumn likelihood estimators (Theorem 5.1 p. 463 of \cite{B-LehCas98}) shows that asymptotic normality result holds. The verification of the assumptions (A-D) of \cite{B-LehCas98}, p.462-463 are omitted; we just focus, here, on the computation of the Fisher information matrix. The antidiagonal term is given by
$$
J_{12}:= \frac12 E\left[ \frac{\partial}{\partial \mu} d_G(\mu,\theta_\gamma)^2 \right] \Bigg|_{\mu=\mu^\star,\gamma=\gamma^\star} = \frac12 \frac{\partial V}{\partial \mu}(\mu^\star,\gamma^\star)=0,
$$
since $\mu^\star$ corresponds to the intrinsic mean and thus minimizes the geodesic variance. The asymptotic variance of $\sqrt{n}\widehat{\gamma}^{MLE}$ is  given by the inverse of
$$
J_2(\gamma^\star)= \frac{k^{\prime\prime}(\gamma^\star)k(\gamma^\star)-k^\prime(\gamma^\star)^2 }{k(\gamma^\star)^2}.
$$
Recall that from (\ref{dg4}), this constant is positive. Now, the last term to compute is the asymptotic variance of $\sqrt{n}\widehat{\mu}^{MLE}$ given by the inverse of
$$
J_1(\gamma^\star):= \frac{\gamma}2 E\left[ \frac{\partial^2}{\partial\mu^2} d_G(\mu,\theta_\gamma)^2\right] \Bigg|_{\mu=\mu^\star,\gamma=\gamma^\star}
 = \frac{\gamma^\star}2 \frac{\partial^2 V}{\partial \mu^2}(\mu^\star,\gamma^\star).$$
From Lemma~\ref{lem-sig}, $\sigma_I^2(\mu,\gamma)$ is a function of $\delta=\mu^\star-\mu$. Without loss of generality, assume $\delta \geq 0$ (the other case leads to the same conclusion), then the function $g$ (in Lemma~\ref{lem-sig}) is twice continuous differentiable on $[0,\pi)$ and $g^{\prime \prime}(\delta) = 2 -4\pi f(\pi-\delta)$. Setting $\delta=0$ in the last equation leads to the stated result. 
\end{proof}


\subsection{Simulation study}

We have investigated the efficiency of the maximum likelihood estimates in a simulation study. A part of the results are presented in Table~\ref{tab-sim}. As expected, the empirical MSE of both estimates of the parameters $\mu^\star$ and $\gamma^\star$ converge towards zero as the sample size grows. We also notice that it's more complicated to estimate the intrinsic mean when the concentration parameter is low. Unlike this, the concentration parameter is better estimated for low values of $\gamma^\star$. These facts are confirmed by Figure~\ref{fig-vMLE} which shows the constants of the asymptotic variances (for both estimates), i.e. $1/J_1(\gamma^\star)$ and $1/J_2(\gamma^\star)$,  in terms of $\gamma^\star$. Figure~\ref{fig-tcl} illustrates the central limit theorem satisfied by the MLE estimates.

\begin{table}[H]
\begin{center}
\begin{tabular}{l|ccccc}
\hline
\multicolumn{6}{c}{Location parameter $\mu^\star$} \\
\hline
Simulation& \multicolumn{5}{c}{Sample size} \\
Parameters &  $n=10$ & $n=20$ & $n=50$ & $n=100$ & $n=500$ \\
  \hline
$\mu^\star=\frac{\pi}4, \gamma^\star=.5$ & 3.3668 & 1.6515 & 0.2218 & 0.0313 & 0.0048 \\ 
 $\mu^\star=\frac{3\pi}4, \gamma^\star=1$   & 0.1088 & 0.0537 & 0.0206 & 0.0103 & 0.0020 \\ 
$\mu^\star=\frac{5\pi}4, \gamma^\star=5$    &  0.0199 & 0.0102 & 0.0040 & 0.0019 & 0.0004\\ 
 $\mu^\star=\frac{7\pi}4, \gamma^\star=10$   & 0.0099 & 0.0050 & 0.0020 & 0.0010 & 0.0002 \\ 
   \hline
\end{tabular}
\vspace*{1cm}

\begin{tabular}{l|ccccc}
\hline
\multicolumn{6}{c}{Concentration parameter $\gamma^\star$} \\
\hline
Simulation& \multicolumn{5}{c}{Sample size} \\
Parameters & $n=10$ & $n=20$ & $n=50$ & $n=100$ & $n=500$ \\ 
  \hline
$\mu^\star=\frac{\pi}4, \gamma^\star=.5$ & 0.2893 & 0.0748 & 0.0255 & 0.0141 & 0.0065  \\ 
  $\mu^\star=\frac{3\pi}4, \gamma^\star=1$ &  1.0559 & 0.2268 & 0.0565 & 0.0252 & 0.0046\\ 
 $\mu^\star=\frac{5\pi}4, \gamma^\star=5$   & 22.8649 & 5.4092 & 1.3365 & 0.5757 & 0.1037 \\ 
 $\mu^\star=\frac{7\pi}4, \gamma^\star=10$   &  88.5093 & 20.1554 & 5.3607 & 2.2723 & 0.4179\\ 
   \hline
\end{tabular}
\end{center}
\caption{Empirical Mean Squared Error (MSE) of MLE estimates of the location parameter $\mu^\star$ (top) and the concentration parameter $\gamma^\star$ (bottom) based on $m=5000$ replications of $gN$ distributions for differents choices of parameters and different sample sizes.}\label{tab-sim}
\end{table}

\bibliographystyle{alpha}

\bibliography{circle}

\begin{figure}[H]
\begin{center}
\begin{tabular}{cc}
\includegraphics[scale=.4]{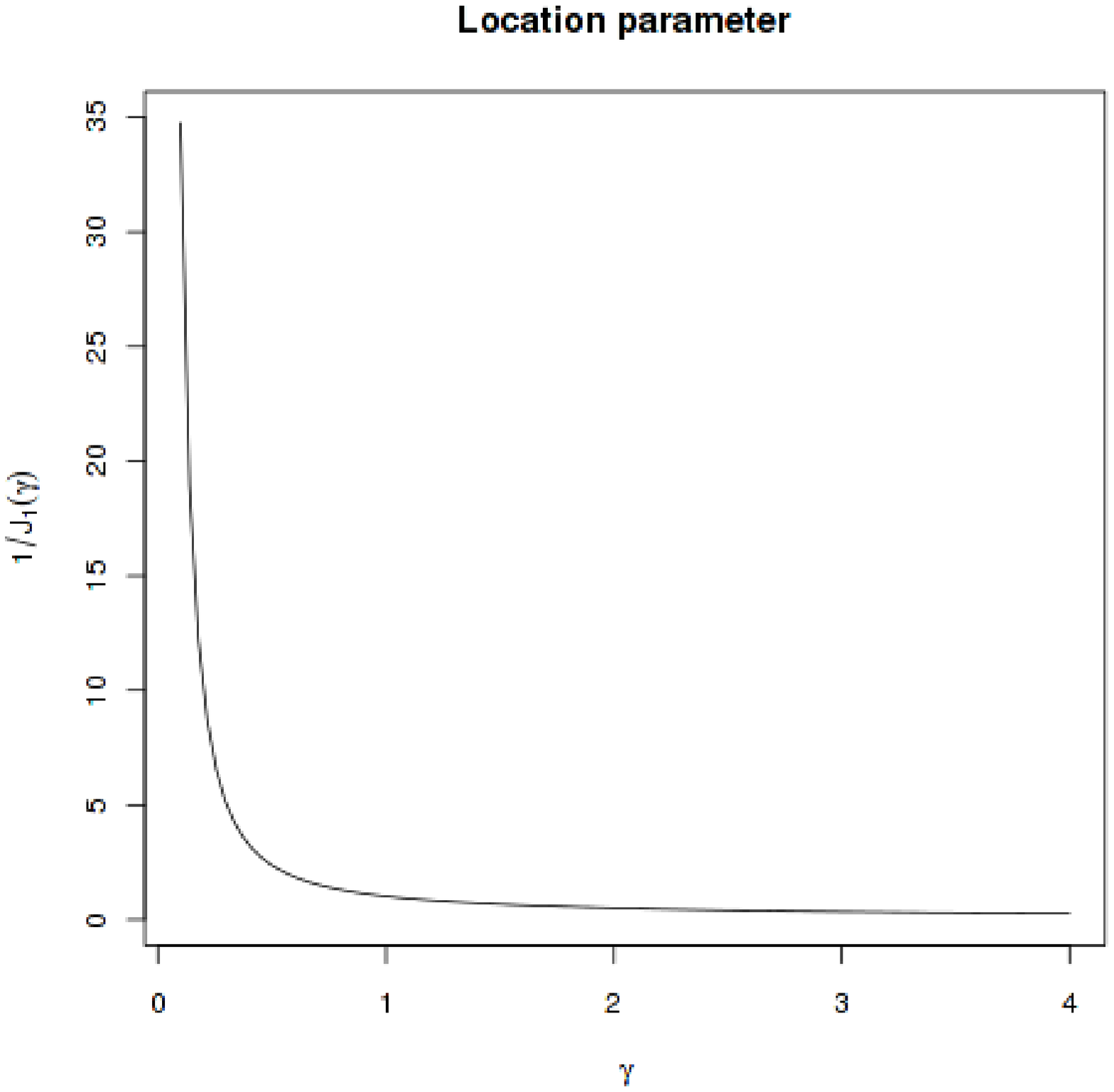} & \includegraphics[scale=.4]{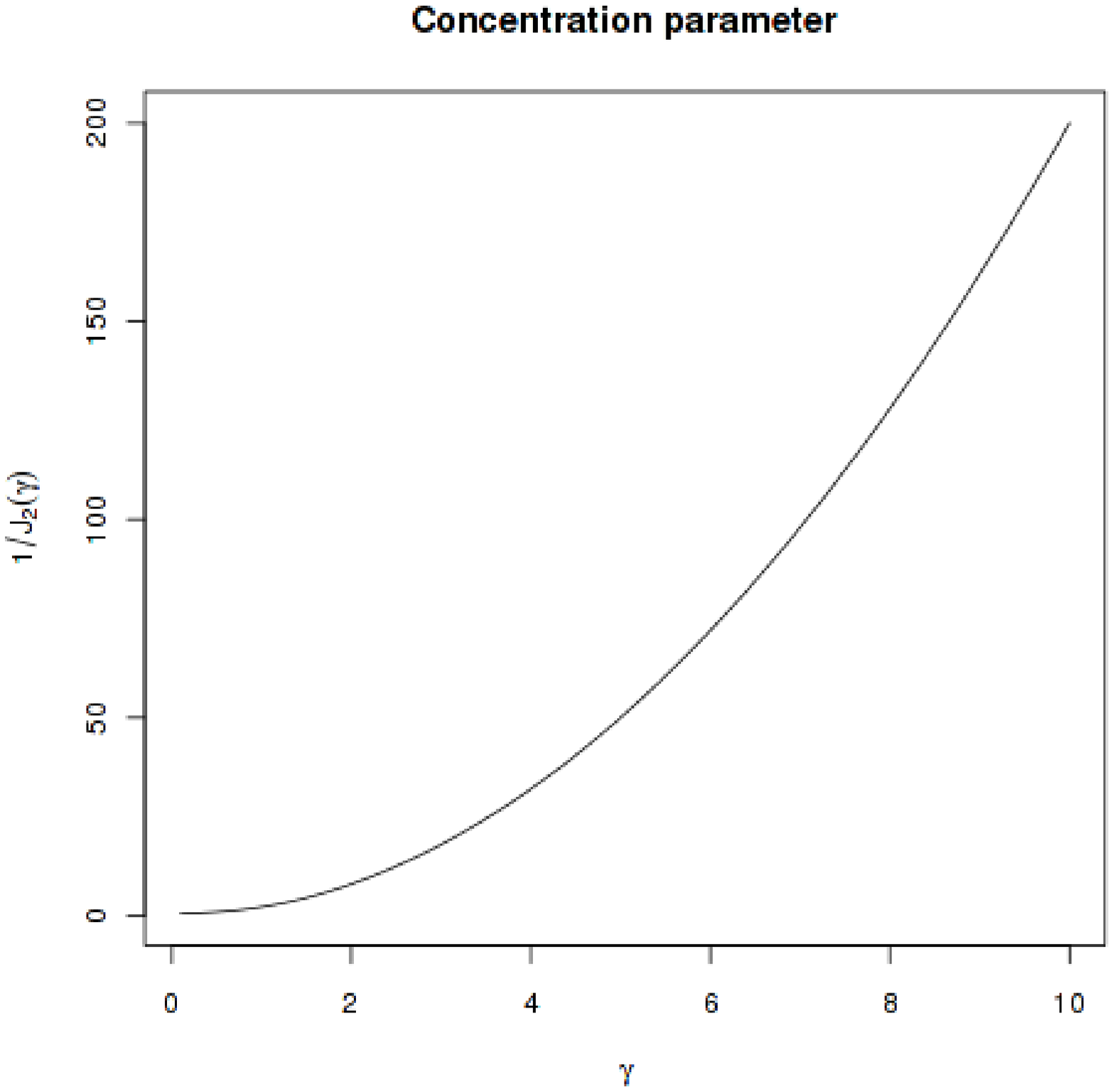}
\end{tabular}
\caption{Plots of the constants of the asymptotic variances of $\widehat{\mu}^{MLE}$ (left) and $\widehat{\gamma}^{MLE}$ (right), i.e. the constants $1/J_1(\gamma^\star)$ and $1/J_2(\gamma^\star)$ given in Proposition~\ref{prop-mle}, in terms of $\gamma^\star$.}\label{fig-vMLE}
\end{center}
\end{figure}

\begin{figure}[H]
\begin{center}
\includegraphics[scale=.8]{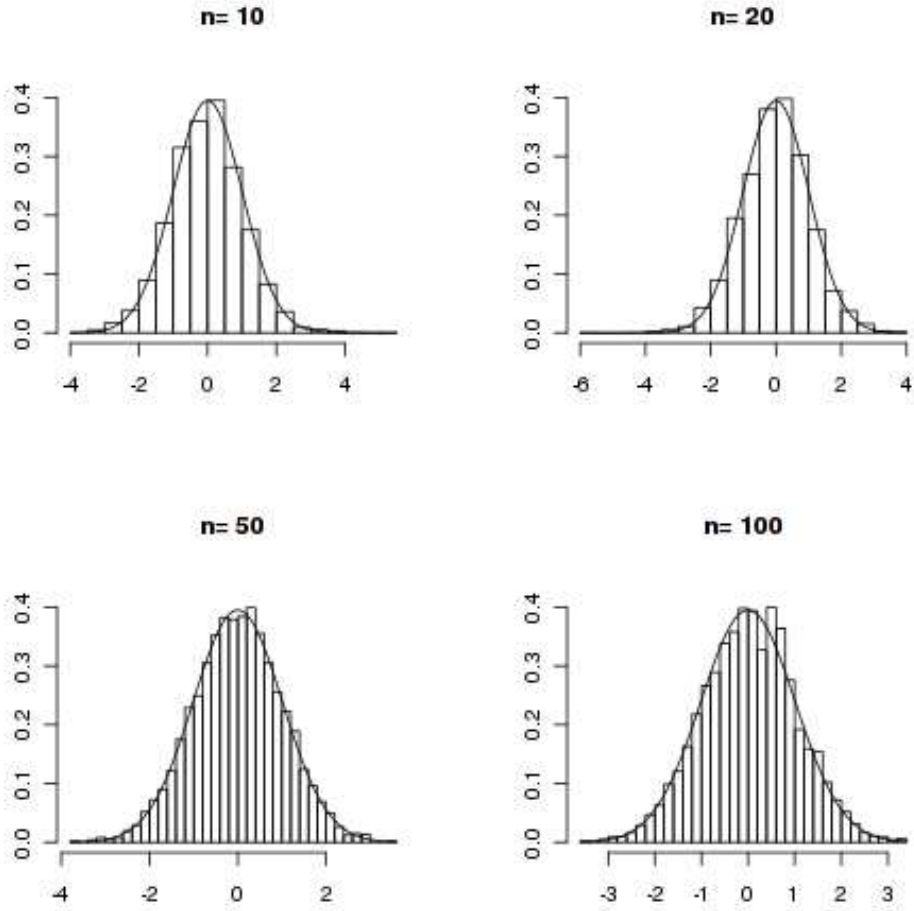}
\caption{Histograms of $\left( \sqrt{n}(\widehat{\mu}^{MLE}_j-\mu^\star)\right)_{j=1,\ldots,5000}$ based on $m=5000$ replications of a $gN(3\pi/4,1)$ distribution for different sample sizes. The curve corresponds to the density of a Gaussian random variable with mean zero and variance $1/J_1(1)$ given in Proposition~\ref{prop-mle}.}\label{fig-tcl}
\end{center}
\end{figure}

\end{document}